\documentclass[letterpaper, 10 pt, conference]{ieeeconf}  

\IEEEoverridecommandlockouts                              
\usepackage{cite}
\usepackage{amsmath,amssymb,amsfonts}
\usepackage{algorithmic}
\usepackage{graphicx}
\usepackage{textcomp}
\usepackage{xcolor}
\usepackage{graphics}
\usepackage{epsfig}
\usepackage{subfig}
\usepackage{float}
\usepackage{color}
\usepackage{cancel}
\usepackage{subfig}
\usepackage{cuted} 
\usepackage{amsfonts}
\usepackage{epstopdf}
\usepackage{amssymb}
\usepackage{mathptmx}
\usepackage{ntheorem}
\usepackage{cleveref}

\newtheorem{lemma}{Lemma}
\newtheorem{thme}{Theorem}
\newtheorem{remark}{Remark}
\newtheorem{define}{Definition}

\title{\LARGE \bf
Exponential Stabilization of Moving Shockwave in  ARZ Traffic Model via Boundary Control:  \emph{Explicit Gains and Arbitrary  Decay Rate}
}

\author{Mina Cao,  Mamadou Diagne, Peipei Shang and Lei Yu 
\thanks{The work of Mina Cao, Peipei Shang and Lei Yu was funded by the National Natural Science Foundation of China (No. 12171368). The work of Mamadou Diagne was funded by the NSF CAREER Award CMMI-2302030 and  the NSF grant CMMI-2222250.  }
\thanks{Mina Cao (2230913@tongji.edu.cn), Peipei Shang (shang@tongji.edu.cn)  and Lei yu (yu\_lei@tongji.edu.cn) are with School of Mathematical Sciences, Key Laboratory of Intelligent Computing and Applications (Ministry of Education), Tongji University.  Mamadou Diagne (mdiagne@ucsd.edu) is with Department of Mechanical and Aerospace Engineering, UC San Diego, 9500 Gilman Drive, La Jolla, CA, 92093-0411.}
}

\allowdisplaybreaks
\begin{document}
\setlength\abovedisplayskip{2pt}
        \setlength\belowdisplayskip{2pt}
        \setlength\abovedisplayshortskip{2pt}
        \setlength\belowdisplayshortskip{2pt}
        \allowdisplaybreaks
        \setlength{\parindent}{1em}
        \setlength{\parskip}{-0em}  
        \addtolength{\oddsidemargin}{2 pt}

\maketitle
\thispagestyle{empty}
\pagestyle{empty}

\begin{abstract}

This paper develops boundary feedback controls to stabilize traffic congestion toward a predefined shock equilibrium in the Aw-Rascle-Zhang (ARZ) traffic flow model. We transform the corresponding moving-boundary $2\times2$ hyperbolic system, covering free and congested flow regimes, respectively, into a shock-free $4\times4$ augmented system on a fixed domain via shock-location-based moving coordinates. By applying the modified Lyapunov functionals concerning shock perturbation, we show that the shock position and the state of the system in $H^2$-norm can be stabilized with an arbitrary exponential decay rate via the given feedback controls. Finally, the stabilization results are demonstrated by numerical simulations.

\end{abstract}

\section{INTRODUCTION}

To overcome the drawbacks of the LWR model \cite{lwr,richa}, various studies have introduced an additional momentum conservation law to describe traffic flow acceleration, giving rise to the class of second-order macroscopic traffic flow models. The works in \cite{ar} and \cite{zg} independently developed representative models of this class. These two models are mathematically equivalent under some conditions (see \cite{YU2019}), yielding the unified ARZ model studied in this paper, which is formulated as a $2\times2$ nonlinear conservation laws:
\begin{equation}
	\begin{cases}
		\partial_t\rho+\partial_x(z-\rho p(\rho))=0,\\
		\partial_tz+\partial_x\left(\frac{z^2}\rho-zp(\rho)\right)=0,
	\end{cases}
	\label{arz}
\end{equation}
where, $\rho$ denotes the traffic density, $p(\rho)$ denotes the traffic pressure, and $z$ is the generalized momentum
defined as $z=\rho(v+p(\rho))$, where $v$ is the flow velocity.
The first equation is conservation of mass and the second one is conservation of generalized momentum. 

Most existing studies on the stability of hyperbolic traffic flow systems focus on smooth classical solutions, yet shock waves naturally emerge in practical traffic flow, as characteristic speeds rely on local traffic states such as density and velocity. Traffic congestion typically propagates in a pattern analogous to shock waves, making it critical to study boundary feedback control problems involving shock waves, with the goal of designing control to stabilize traffic flow and suppress shock wave propagation.
Limited research has focused on shock stabilization for traffic flow models. One main reason is that it is rather difficult to describe the boundary condition when shock solutions are considered for the hyperbolic systems, which makes the stabilization through boundary control much challenging.

Recently, \cite{burger} realized shock steady-state stability for the Burgers equation via boundary feedback control using shock position and bilateral state information, by transforming the shock-involved system into a shock-free system via coordinate transformation and introducing a shock-dependent perturbation term. This framework has also been extended to the Saint-Venant equations for hydraulic jump control \cite{sv}. The method they used is the Lyapunov approach, initially proposed by Coron {\it et al.}\cite{coron1st1} and then adapted to various kinds of problems related to stabilization of hyperbolic systems, see \cite{BCBook} for a comprehensive study. Using this method, {Zhang et al. \cite{stopandgo} began to study a very relevant stop-and-go wave stabilization for the ARZ model and gave sufficient conditions for local exponential stabilization} in terms of matrix inequalities; however, it is usually difficult to prove theoretically the existence of the parameters using LMIs, only numerical method can be applied.

On the other hand, based on the LWR model, \cite{bstp} studied stabilization of traffic flows when a shock has formed.
Taking the system as a delayed control system, they achieved local stability with an arbitrary convergence rate. The main method they used is the backstepping method, initially proposed by Krstic \cite{krstic20082}, is a powerful tool to study the stabilization of hyperbolic systems with source terms with less controls, see \cite{KrsticBook} for a comprehensive introduction.
However, the feedback control obtained by the backstepping method is full-state feedback, which is usually difficult to implement, although sometimes the observer can be designed \cite{2017DDTK}.

Inspired by \cite{burger}, in this paper, we target shock-based traffic congestion stabilization via straightforward static-feedback control by the Lyapunov approach, i.e., the boundary feedback control we apply only needs the information from the boundaries, which is relatively simple to apply in practical applications.
To that end, we transform the corresponding moving-boundary $2\times2$ hyperbolic system, covering free and congested flow regimes, respectively, into a shock-free $4\times4$ augmented system on a fixed domain via shock-location-based moving coordinates.
Different from the work in \cite{stopandgo}, where the system is formulated in terms of the evolution of the density and the velocity, in this work, we consider the density and the generalized momentum as two states. The reason is that the two quantities are conserved through the jump in the jump's reference frame.
{ In addition, we emphasize the main differences between the work \cite{stopandgo} and our results. Firstly, the authors assume in \cite{stopandgo} that the traffic pressure is linear with respect to the density $\rho$ and has a very special precise form, which results in a relatively simple structure of hyperbolic systems and makes the analysis of dissipative boundary conditions easier to perform. By contrast, in this paper, we deal with general traffic pressure satisfying reasonable conditions, which include the special case studied in \cite{stopandgo}. 
Secondly, by fully using the information contained in the shock, we can realize the exponential stabilization using only three controls rather than four controls that are designed in \cite{stopandgo}.
Thirdly, by choosing the feedback controls properly, we can prove theoretically the existence of the tuning parameters to stabilize the system with a shock and give an explicit range of the tuning parameters. Moreover, the decay rate can be made arbitrarily large.
}

The remainder of this paper is structured as follows. Section \ref{ps} states the stabilization problem of the ARZ model with shock and proposes boundary feedback control strategy. Section \ref{wps} establishes the well-posedness of the system. Section \ref{stb-pf} proves the exponential stability of the closed-loop system. Numerical simulations in Section \ref{sec_simu} verify the theoretical results. Section \ref{sec_con} concludes the paper.

\section{Problem statement}
\label{ps}
Consider a road segment of length $L$, that is, the space domain involved in the problem is bounded, denoted as $[0,L]$. We focus on the stabilizaiton of the ARZ model (\ref{arz}) with shock in this space domain. 
The pressure $p\in C^2([0,+\infty);[0,+\infty))$ satisfies the following assumptions:
\begin{equation}
		p(0)=0,\quad
		p'(\rho)>0,\quad \rho\to \rho p(\rho)\,\, \text{is strictly convex}. 
	\label{passume}
\end{equation}
To close the system we need to establish a relationship between $\rho$ and $z$ before and after the shock, since the solution $(\rho,z)^\top$ of (\ref{arz}) exists a jump discontinuity on $x_s(t)\in(0,L)$. For any well-defined function $f$ in the neighborhood of $x_s$, the following notation is given:
$$[f]_-^+=f(x_s^+(t))-f(x_s^-(t)).$$
By Rankine-Hugoniot condition we have,
\begin{equation}
	\begin{cases}
		[z-\rho p(\rho)]_-^+=\dot{x_s}[\rho]_-^+,\\
		\left[\frac{z^2}\rho-zp(\rho)\right]_-^+=\dot{x_s}[z]_-^+,
	\end{cases}
	\label{rh1}
\end{equation}
where $\dot{x_s}$ denotes the time derivative of $x_s$, which is the moving pace of the jump discontinuity. From \eqref{rh1}, direct computations give that
\begin{align}
	&\dot{x_s}=\frac{[z-\rho p(\rho)]_-^+}{[\rho]_-^+},\label{sh}\\
    &\left[\frac z\rho\right]_-^+\left[\frac z\rho - p(\rho)\right]_-^+=0.\label{rh}
    \end{align}
One can easily check that the first characteristic field is genuinely nonlinear, the second is linearly degenerate.
The points $(\rho_f,z_f)\in (\mathbb{R}^+)^2$ and $(\rho_c,z_c)\in (\mathbb{R}^+)^2$
are joined by a shock $x_s(t)$ if and only if the following holds 
\begin{equation}\label{e:rhc}
    \frac{z_f(t,x_s^-)}{\rho_f(t,x^-_s)}= \frac{z_c(t,x_s^+)}{\rho_c(t,x^+_s)}, \quad \rho_f(t,x^-_s)>\rho_c(t,x^+_s),
\end{equation}
see \cite{VARZ} for the details.
Our aim is to stabilize the shock steady state $\left((\rho^*,z^*)^\top,x_s^*\right)$ of the system of (\ref{arz}) and \eqref{sh}--(\ref{rh}). The steady state satisfies $x_s^*\in(0,L)$ and
	\begin{equation}
		\rho^*\!\!=\!\!\begin{cases}
			\rho_f^*,\, x\in[0,x_s^*),\\
			\rho_c^*,\, x\in(x_s^*,L],
		\end{cases}
        z^*\!\!=\!\!\begin{cases}
			z_f^*,\, x\in[0,x_s^*),\\
			z_c^*,\, x\in(x_s^*,L],
		\end{cases}
		\label{steady}
	\end{equation}
	where $\rho_f^*,\ \rho_c^*,\ z_f^*,\ z_c^*$ are positive constants satisfying
\begin{equation}\label{czr}
       z^*_f \rho^*_c=z^*_c\rho^*_f.
    \end{equation}
By (\ref{sh}), $\dot{x_s}=0$ implies that the steady state $\big(\left(\rho^*,z^*\right)^\top,x_s^*\big)$ satisfies
	$$z_f^*-\rho_f^*p\big(\rho_f^*\big)=z_c^*-\rho_c^*p\left(\rho_c^*\right).$$
	The traffic flow is free before the shock and congested after the shock. Thus the eigenvalues of the system are both positive before the discontinuity, i.e., when $x\in[0,x_s^*)$
    \begin{equation}\label{eig1}
    \begin{aligned}
    &\lambda_1=-p(\rho_f^*)+\frac{z_f^*}{\rho_f^*}-\rho_f^*p'(\rho_f^*)>0,\\
        & \lambda_2=-p(\rho_f^*)+\frac{z_f^*}{\rho_f^*}>0,
\end{aligned}
\end{equation}
	while become one positive and one negative after the shock, i.e., when $x\in(x_s^*,L]$
	\begin{equation}\label{eig2}
    \begin{aligned}
		-\lambda_3&=-p(\rho_c^*)+\frac{z_c^*}{\rho_c^*}-\rho_c^*p'(\rho_c^*)<0, \\
        \lambda_4&=-p(\rho_c^*)+\frac{z_c^*}{\rho_c^*}>0.
        \end{aligned}
	\end{equation}
Note that by the assumption of traffic pressure (\ref{passume}) and \eqref{czr}, equations (\ref{eig1})--(\ref{eig2}) implies that $\rho_f^*<\rho_c^*$.

We consider the following static feedback control:
\begin{equation}
	\begin{pmatrix}
		\rho(t,0)\!-\!\rho_f^*\\
		z(t,0)\!-\!z_f^*\\
		z(t,L)\!-\!z_c^*
	\end{pmatrix}\!\!=\!G\!\!\begin{pmatrix}
		z(t,x_s^-)\!-\!z_f^*\\
		z(t,x_s^+)\!-\!z_c^*\\
		\rho(t,x_s^-)\!-\!\rho_f^*\\
		x_s\!-\!x_s^*
	\end{pmatrix}\!\!-\!\!\begin{pmatrix}
		0\\
		0\\
		G_4(\rho(t,L)\!-\!\rho_c^*)
	\end{pmatrix},
	\label{boun}
\end{equation}
where $G=(G_1,G_2,G_3)^\top :\ \mathbb{R}^4\to\mathbb{R}^3,\ G\in C^2$, $G_4 :\ \mathbb{R}\to\mathbb{R},\ G_4\in C^2$, and
\begin{equation}
	G(\mathbf0)=\mathbf0,\quad G_4(0)=0,\quad G'_4(0)=-\lambda_4.
	\label{boun11}
\end{equation}
Note that the steady state$\left((\rho^*,z^*)^\top,x_s^*\right)$ satisfies the boundary conditions (\ref{boun}) obviously. In real case applications, the implementation of this boundary feedback control is rather easy since it only needs some pointwise measurements at $\rho(t,L), x_s(t), \rho(t,x_s^-), z(t,x_s^+)$ and $z(t,x_s^-)$.

Consider the following initial conditions that
\begin{equation}
	\rho(0,x)=\rho^0(x),\quad z(0,x)=z^0(x),\quad x_s(0)=x_s^0,
	\label{ini}
\end{equation}
where $x_s^0\in(0,L)$ and $(\rho^0(x),z^0(x))^\top\!\!\!\!\in\!\!\! H^2\left((0,x_s^0);\mathbb{R}^2\right)\!\!\cap\!\! H^2\left((x_s^0,L);\mathbb{R}^2\right).$ {\color{black}Suppose that the initial conditions satisfy the first-order compatibility condition implied by \eqref{boun} and \eqref{ini},} we make a definition as follows:
\begin{define}
	If there exists $\delta^*>0,\ C^*>0,$ s.t. $\forall\ (\rho^0(x),z^0(x))^\top\in H^2\left((0,x_s^0);\mathbb{R}^2\right)\cap H^2\left((x_s^0,L);\mathbb{R}^2\right),$ $\forall\ x_s^0\in(0,L)$ satisfy
	\begin{equation*}
    \begin{aligned}
		&\left\Vert(\rho^0-\rho_f^*,z^0-z_f^*)^\top\right\Vert_{H^2\left((0,x_s^0);\mathbb{R}^2\right)}\\
        +&\left\Vert(\rho^0-\rho_c^*,z^0-z_c^*)^\top\right\Vert_{H^2\left((x_s^0,L);\mathbb{R}^2\right)}+\left|x_s^0-x_s^*\right|\leqslant\delta^*,
        \end{aligned}
	\end{equation*}
	and the first-order compatibility condition given by \eqref{boun}, also for any $T>0$, the system of $(\ref{arz})$, \eqref{sh}--\eqref{rh}, \eqref{boun}, $(\ref{ini})$ has a unique solution $x_s\in C^1([0,T])$ and $(\rho,z)^\top\in C^0\left([0,T];H^2\left((0,x_s(t));\mathbb{R}^2\right)\cap H^2\left((x_s(t),L);\mathbb{R}^2\right)\right)$
    satisfying
	\begin{equation}
		\begin{aligned}
		&\left\Vert(\rho(t,\cdot)-\rho_f^*,z(t,\cdot)-z_f^*)^\top\right\Vert_{H^2\left((0,x_s(t));\mathbb{R}^2\right)}\\
        +&\left\Vert(\rho(t,\cdot)-\rho_c^*,z(t,\cdot)-z_c^*)^\top\right\Vert_{H^2\left((x_s(t),L);\mathbb{R}^2\right)}\\
        +&\left|x_s(t)-x_s^*\right|\\
			\leqslant &C^*e^{-\gamma t}\Big(\left\Vert(\rho^0-\rho_f^*,z^0-z_f^*)^\top\right\Vert_{H^2\left((0,x_s^0);\mathbb{R}^2\right)}\\
            &+\left\Vert(\rho^0-\rho_c^*,z^0-z_c^*)^\top\right\Vert_{H^2\left((x_s^0,L);\mathbb{R}^2\right)}+\left|x_s^0-x_s^*\right|\Big)
		\end{aligned}
		\label{stbcon3}
	\end{equation}
	for any $t\in[0,T]$, then we say the steady state $\left((\rho^*,z^*)^\top,x_s^*\right)$ is locally exponentially stable with decay rate $\gamma$ in $H^2$ norm.
	\label{stbdef}
\end{define}
Before introducing the main stability result in this article, we give the following notations and definitions:
\begin{equation*}
\begin{aligned}
	&D(x,\gamma)\!=\!\mbox{diag}\left(\!\!\frac{s_i}{b_i}\!\left(\textcolor{black}{\frac{\lambda_i}{\widetilde\lambda_i}}\!-\!\textcolor{black}{\frac{s_i\lambda_4}{\widetilde \lambda_4}}R_i\right)\!\exp\!\left(\!\!\frac{\gamma\left(x_s^*-x_s\right)}{x_i\lambda_i}\!\!\right)\!\!\right)_{i\in\{1,2,3\}},
    \end{aligned}
\end{equation*}

\begin{equation*}
\begin{aligned}
	&\widetilde{D}(\mu)\!=\! \mbox{diag}\! \left(\! \!\left(\! \sum_{j=1}^3\!\exp\!\left(\! \frac{\mu x_s^*}{x_i\lambda_i}\! \!-\!\! \frac{\mu x_s^*}{x_j\lambda_j}\! \right)\! \right)\! \left(\textcolor{black}{\!\!\frac{\lambda_i}{\widetilde\lambda_i}}\!\!-\!\!\textcolor{black}{\frac{s_i\lambda_4}{\widetilde \lambda_4}}R_i\!\!\right)^2\right)_{i \in\{1,2,3\}}.
    \end{aligned}
\end{equation*}
Let us denote by
\begin{equation*}
K_1=\begin{pmatrix}
		\frac{\widetilde\lambda_1\widetilde\lambda_2}{\widetilde\lambda_2-\widetilde\lambda_1}&\frac{-\widetilde\lambda_1}{\widetilde\lambda_2-\widetilde\lambda_1}&0\\
		\frac{-\widetilde\lambda_1\widetilde\lambda_2}{\widetilde\lambda_2-\widetilde\lambda_1}&\frac{\widetilde\lambda_2}{\widetilde\lambda_2-\widetilde\lambda_1}&0\\
		0&0&\frac{\widetilde\lambda_3}{\widetilde\lambda_3+\widetilde\lambda_4}
	\end{pmatrix}
\end{equation*}
and define \begin{equation}
	K=K_1G'(\mathbf{0})\begin{pmatrix}
		1&1&0\\
		\frac{\rho_c^*-\frac{z_c^*}{\widetilde\lambda_1}}{\rho_f^*-\frac{z_f^*}{\widetilde\lambda_4}}&\frac{\rho_c^*-\frac{z_c^*}{\widetilde\lambda_2}}{\rho_f^*-\frac{z_f^*}{\widetilde\lambda_4}}&-\frac{\frac{z_f^*}{\widetilde\lambda_3}+\frac{z_f^*}{\widetilde\lambda_4}}{\rho_f^*-\frac{z_f^*}{\widetilde\lambda_4}}\\
		\frac{1}{\widetilde\lambda_1}&\frac{1}{\widetilde\lambda_2}&0\\
		0&0&0
	\end{pmatrix},
	\label{stbnt-K}
\end{equation}

\begin{equation}
	(b_1,b_2,b_3)^\top=K_1G'(\mathbf{0})(0,0,0,1)^{\top},
	\label{stbnt-b}
\end{equation}
where $s_1=s_2=-s_3=1$, $x_1=x_2=1,\ x_3=-x_4=\frac{x_s^*}{L-x_s^*},$
$\widetilde\lambda_i=s_i\lambda_i+p(\rho^*_{(i)})+\rho^*_{(i)}p'(\rho^*_{(i)})$ and
\begin{equation}\label{notations}
\begin{aligned}
\begin{pmatrix}
	\rho_{(i)}^*\\
	z_{(i)}^*
\end{pmatrix}&=\begin{cases}
\left(\rho_f^*,z_f^*\right)^\top,\quad &i\! =\! 1,2,\\
\left(\rho_c^*,z_c^*\right)^\top,\quad &i\! =\! 3,4,
\end{cases},\\
R_i&=\frac{\rho_{(5-i)}^*-\frac{z_{(5-i)}^*}{\widetilde\lambda_i}}{\rho_f^*-\frac{z_f^*}{\widetilde\lambda_4}},\ i=1,2,3.
\end{aligned}
\end{equation}
On the basis of Definition \ref{stbdef}, we have the main results:
\begin{thme}
	Consider system $(\ref{arz})$ with boundary conditions $(\ref{boun})$, for any given steady state $\left((\rho^*,z^*)^\top,x_s^*\right)$, denote by $$L_i=\frac{1}{\rho_c^*-\rho_f^*}\left(\textcolor{black}{-\frac{\lambda_i}{\widetilde\lambda_i}}+\textcolor{black}{\frac{s_i\lambda_4}{\widetilde \lambda_4}}R_i\right),\ i=1,2,3.\ $$ For any $\gamma>0$, suppose $b_i,\ i=1,2,3$ satisfy
	\begin{itemize}
		\item If $L_i>0$, then
		\begin{equation}
			b_i\!\!\in\!\!\left(\!\!\frac{-\gamma\exp\left(\!-\frac{\gamma x_s^*}{x_i\lambda_i}\right)}{3L_i\left(1\!\!-\!\exp\left(\!-\frac{\gamma x_s^*}{x_i\lambda_i}\right)\right)},\frac{-\gamma\exp\left(\!-\frac{\gamma x_s^*}{x_1\lambda_1}\right)}{3L_i}\!\!\right),
			\label{bi+}
		\end{equation}
		\item If $L_i<0$, then
		\begin{equation}
			b_i\!\!\in\!\!\left(\!\!\frac{-\gamma\exp\left(\!-\frac{\gamma x_s^*}{x_i\lambda_i}\right)}{3L_i},\frac{-\gamma\exp\left(-\frac{\gamma x_s^*}{x_i\lambda_i}\right)}{3L_i\left(1\!\!-\!\exp\!\left(\!-\frac{\gamma x_s^*}{x_1\lambda_1}\right)\!\right)}\!\right),
			\label{bi-}
		\end{equation}
	\end{itemize}
	and the matrix
	\begin{equation}
		\begin{aligned}
			&D\left(x_s^*,\gamma\right)-K^\top D\left(0,\gamma\right)K-\\
			&\left(\sum_{k=1}^3\!\frac{-2b_ks_k}{\gamma^2\left(\rho_c^*\!-\!\rho_f^*\right)^2}\!\left(\textcolor{black}{\frac{\lambda_i}{\widetilde\lambda_i}}\!\!-\!\!\textcolor{black}{\frac{s_i\lambda_4}{\widetilde \lambda_4}}R_i\right)\left(\!\exp\!\left(\!\frac{\gamma x_s^*}{x_k\lambda_k}\!\right)\!\!-\!1\!\!\right)\!\!\!\right)\!\!\widetilde{D}(\gamma)
		\end{aligned}
		\label{+dft}
	\end{equation}
	is positive definite, then the steady state $\left((\rho^*,z^*)^\top,x_s^*\right)$ is locally y exponentially stable with the decay rate of $\frac{\gamma}4$ in $H^2$ norm.
	\label{stbthm}
\end{thme}
The existence of the control function $G$
such that $K$ and $(b_1, b_2, b_3)^{\top}$ defined in (\ref{stbnt-K})--(\ref{stbnt-b}) satisfying (\ref{bi+})--(\ref{+dft}) is discussed in
Remark \ref{exK}.
\section{Well-posedness of the system}
\label{wps}
In this section, we will prove the well-posedness of the system of the ARZ model (\ref{arz}) with shock (\ref{rh}), boundary feedback control (\ref{boun}) and the initial condition (\ref{ini}). Firstly, we state the following well-posedness theorem.
\begin{thme}\label{thm02}
	For any $T>0,\ \exists\ \delta(T)>0$ such that for any given initial condition $(\ref{ini})$ satisfies
	\begin{equation}\label{wpdcon1}
    \begin{aligned}
		&\left\Vert(\rho^0-\rho_f^*,z^0-z_f^*)^\top\right\Vert_{H^2\left((0,x_s^0);\mathbb{R}^2\right)}\\
        &+\left\Vert(\rho^0-\rho_c^*,z^0-z_c^*)^\top\right\Vert_{H^2\left((x_s^0,L);\mathbb{R}^2\right)}+\left|x_s^0-x_s^*\right|\leqslant\delta(T),
        \end{aligned}
	\end{equation}
	and the first-order compatibility conditions, the system (\ref{arz}),\eqref{sh}--\ref{rh}), (\ref{boun}),(\ref{ini}) has a unique solution $$(\rho,z)^\top\!\!\in\!C^0\!\left([0,T];H^2\left((0,x_s(t));\mathbb{R}^2\right)\!\cap\! H^2\left((x_s(t),L);\mathbb{R}^2\right)\right),$$
	$$x_s\in C^1([0,T]),$$ and the following inequality
	\begin{equation*}
		\begin{aligned}
			&\hspace{-1em}\left\Vert(\rho(t,\cdot)-\rho_f^*,z(t,\cdot)-z_f^*)^\top\right\Vert_{H^2\left((0,x_s(t));\mathbb{R}^2\right)}\\
            \\+ &\left\Vert(\rho(t,\cdot)-\rho_c^*,z(t,\cdot)-z_c^*)^\top\right\Vert_{H^2\left((x_s(t),L);\mathbb{R}^2\right)}\! +\! \left|x_s(t)-x_s^*\right|&\\
			\leqslant &C(T)\Big(\left\Vert(\rho^0-\rho_f^*,z^0-z_f^*)^\top\right\Vert_{H^2\left((0,x_s^0);\mathbb{R}^2\right)}\\
            &+\left\Vert(\rho^0-\rho_c^*,z^0-z_c^*)^\top\right\Vert_{H^2\left((x_s^0,L);\mathbb{R}^2\right)}+\left|x_s^0-x_s^*\right|\Big)&
		\end{aligned}
	\end{equation*}
	holds for any $t\in[0,T].$
	\label{wpdthm}
\end{thme}
\begin{proof}
	We use following change of variables
	\begin{equation}
		\begin{aligned}
			\rho_f(t,x)=&\rho\left(t,x\frac{x_s}{x_s^*}\right),\,z_f(t,x)=z\left(t,x\frac{x_s}{x_s^*}\right),\\
			\rho_c(t,x)=&\rho\left(t,L+x\frac{x_s-L}{x_s^*}\right),\\
            z_c(t,x)=&z\left(t,L+x\frac{x_s-L}{x_s^*}\right),\\
		\end{aligned}
		\label{sbst}
	\end{equation}
	in order to reduce the spatial domain to a fixed interval.
    Define the deviations
	\begin{equation}
		\rho^f\!=\!\rho_f\!-\!\rho_f^*,\ z^f\!=\!z_f\!-\!z_f^*,\ \rho^c\!=\!\rho_c\!-\!\rho_c^*,\ \rho^c\!=\!\rho_c\!-\!\rho_c^*.
		\label{devia}
	\end{equation}
	Then the system (\ref{arz}), \eqref{sh}--(\ref{rh}) becomes equivalent to the $4\times4$ system defined in $\mathbb{R}^+\times[0,x_s^*]$:
	\begin{strip} \begin{equation}
		\begin{cases}
				\partial_t\rho^f\! -\! \left(x\frac{\dot{x_s}}{x_s^*}\! +\! p(\rho^f+\rho_f^*)\! +\! (\rho^f\! +\! \rho_f^*)p'(\rho^f\! +\! \rho_f^*)\right)\frac{x_s^*}{x_s}\partial_x\rho^f\! +\! \frac{x_s^*}{x_s}\partial_xz^f=0,\vspace{0.1cm}\\
				\partial_tz^f\! +\! \left(\frac{2(z^f\! +\! z_f^*)}{\rho^f+\rho_f^*}\! -\! p(\rho^f+\rho_f^*)\! -\! x\frac{\dot{x_s}}{x_s^*}\right)\frac{x_s^*}{x_s}\partial_xz^f\! -\! \left(\frac{(z^f\! +\! z_f^*)^2}{(\rho^f\! +\! \rho_f^*)^2}\! +\! (z^f\! +\! z_f^*)p'(\rho^f\! +\! \rho_f^*)\right)\frac{x_s^*}{x_s}\partial_x\rho^f=0,\vspace{0.1cm}\\
				\partial_t\rho^c\! +\! \left(x\frac{\dot{x_s}}{x_s^*}\! +\! p(\rho^c\! +\! \rho_c^*)\! +\! (\rho^c\! +\! \rho_c^*)p'(\rho^c\! +\! \rho_c^*)\right)\frac{x_s^*}{L\! -\! x_s}\partial_x\rho^c\! -\! \frac{x_s^*}{L\! -\! x_s}\partial_xz^c=0,\vspace{0.1cm}\\
				\partial_tz^c\! -\! \left(\frac{2(z^c\! +\! z_c^*)}{\rho^c\! +\! \rho_c^*}\! -\! p(\rho^c\! +\! \rho_c^*)\! -\! x\frac{\dot{x_s}}{x_s^*}\right)\frac{x_s^*}{L\! -\! x_s}\partial_xz^c\! +\! \left(\frac{(z^c\! +\! z_c^*)^2}{(\rho^c\! +\! \rho_c^*)^2}\! +\! (z^c\! +\! z_c^*)p'(\rho^c\! +\! \rho_c^*)\right)\frac{x_s^*}{L\! -\! x_s}\partial_x\rho^c=0,
			\end{cases}\hspace{-0.4cm}
			\label{4th}
	\end{equation}
   \end{strip}
	where
	\begin{equation}
    \begin{aligned}
		\dot{x_s}=&\frac{1}{{\rho^c+\rho_c^*-\rho^f-\rho_f^*}}\Big[z^c-z^f +z_c^*-\! z_f^*\\
        &-\left(\rho^c+\! \rho_c^*\right)p\left(\rho^c+\rho_c^*\right)+(\rho^f+\! \rho_f^*)p(\rho^f+\rho_f^*)\Big],
		\label{shock1}
        \end{aligned}
	\end{equation}
	the corresponding boundary condition \eqref{e:rhc} at $x^*_s$ becomes
	\begin{equation}
\frac{z^c+z_c^*}{\rho^c+\rho_c^*}=\frac{z^f+z_f^*}{\rho^f+\rho_f^*}.
		\label{boun1}
	\end{equation}
	Next, we introduce the following Riemann coordinates:
	\begin{equation}
		\mathbf{u}=(u_1,u_2,u_3,u_4)^{\top}=\begin{pmatrix}
			S_1&0\\
			0&S_2
		\end{pmatrix}(\rho^f,z^f,\rho^c,z^c)^{\top},
		\label{rco}
	\end{equation}
	where
	$$S_1=\begin{pmatrix}
		\frac1{\widetilde\lambda_1}&\frac1{\widetilde\lambda_2}\\
		1&1
	\end{pmatrix}^{-1},\quad S_2=\begin{pmatrix}
		\textcolor{black}{\frac1{\widetilde\lambda_3}}&\frac1{\widetilde\lambda_4}\\
		1&1
	\end{pmatrix}^{-1}.
	$$
	Then the system (\ref{4th}) can be rewritten as
	\begin{equation}
		\mathbf{u}_t+\left(\Lambda(x_s)+A(\mathbf{u},x_s)+x\dot{x_s}B(x_s)\right)\mathbf{u}_x=0,
		\label{4th2}
	\end{equation}
	where
	\begin{equation}
		\Lambda=\text{diag}\left(\frac{x_s^*}{x_s}\lambda_1,\frac{x_s^*}{x_s}\lambda_2,\frac{x_s^*}{L-x_s}\lambda_3,-\frac{x_s^*}{L-x_s}\lambda_4\right)
		\label{diag}
	\end{equation}
	and $A,\ B\in C^2$ can be acquired by calculation, $A(\mathbf0,x_s)=0.$ The time derivative of shock (\ref{shock1}) becomes
	\begin{align}
    \label{shock2}
		\dot{x_s}=&\frac{1}{\frac{u_4}{\widetilde\lambda_4}\textcolor{black}{+}\frac{u_3}{\widetilde\lambda_3}\!+\!\rho_c^*\!-\!\frac{u_1}{\widetilde\lambda_1}\!-\!\!\frac{u_2}{\widetilde\lambda_2}\!-\!\rho_f^*}\!\Big[u_3\! +\! u_4\! -\! u_1\! -\! u_2\! +\! z_c^*\! -\! z_f^*\! \nonumber\\
        &-\! \left(\frac{u_4}{\widetilde\lambda_4} \textcolor{black}{+}\frac{u_3}{\widetilde\lambda_3}\! +\!\! \rho_c^*\!\!\right) p\! \left(\frac{u_4}{\widetilde\lambda_4} \textcolor{black}{+} \frac{u_3}{\widetilde\lambda_3}\! +\! \rho_c^*\right)\nonumber\\
        &+\! \left(\!\!\frac{u_1}{\widetilde\lambda_1}\! +\! \frac{u_2}{\widetilde\lambda_2}\!\! +\!\! \rho_f^*\!\right) p\! \left(\!\frac{u_1}{\widetilde\lambda_1}\! + \frac{u_2}{\widetilde\lambda_2}\!\! +\!\! \rho_f^*\!\right)\Big].
        \end{align}
By (\ref{rco}), the boundary condition (\ref{boun1}) on $x=x^*_s$ becomes
	\begin{equation}
    \begin{aligned}
		&\rho_f^*(u_3+u_4)-\rho_c^*(u_1+u_2)\\
        +&z_c^*\left(\frac{u_1}{\widetilde\lambda_1}+\frac{u_2}{\widetilde\lambda_2}\right)-z_f^*\left(\frac{u_4}{\widetilde\lambda_4}\textcolor{black}{+}\frac{u_3}{\widetilde\lambda_3}\right)=O(|\mathbf{u}(t,x_s^*)|^2).
		\label{boun3}
        \end{aligned}
	\end{equation}
	The boundary condition (\ref{boun}) can be rewritten as
	\begin{equation}
		\begin{pmatrix}
			u_1(t,0)\\
			u_2(t,0)\\
			u_3(t,0)\end{pmatrix}=\mathcal{B}\begin{pmatrix}
			\begin{pmatrix}
				u_1(t,x_s^*)\\
				u_2(t,x_s^*)\\
				u_3(t,x_s^*)\end{pmatrix},u_4(t,0),x_s-x_s^*\end{pmatrix},
		\label{boun21}
	\end{equation}
	by (\ref{sbst}), (\ref{devia}), (\ref{rco}) and (\ref{boun3}), where $\mathcal{B}=(B_1,B_2,B_3)^\top:\ \mathbb{R}^3\times\mathbb{R}\times\mathbb{R}\to\mathbb{R}^3$, $\mathcal{B}\in C^2$ and
    $\partial_2\mathcal{B}(\mathbf0,0,0)\equiv0$.
Moreover, noticing \eqref{boun11} and using the Implicit Function Theorem,
	\begin{equation}
		B_3=\mathcal{F}\left(u_4(t,0),G_3(\mathbf{u}(t,x_s^*),x_s)\right),
		\label{implicit}
	\end{equation}
	in the neighborhood of $\mathbf{u}=0$ and
	\begin{equation}
		\mathcal{F}(0,0)=0,\,\partial_1\mathcal{F}(0,0)=0,\,\partial_2\mathcal{F}(0,0)=\frac{\widetilde\lambda_3}{\widetilde\lambda_3+\widetilde\lambda_4}.
		\label{boun22}
	\end{equation}
	The initial condition (\ref{ini}) can be transformed as
	\begin{equation}
		\mathbf{u}^0(x)=\left(u_1^0(x),u_2^0(x),u_3^0(x),u_4^0(x)\right)^\top,\,x_s(0)=x_s^0,
		\label{ini2}
	\end{equation}
	satisfying the first-order compatibility condition determined by (\ref{boun21}). Thus, we obtain that the well-posedness of the system (\ref{arz}), \eqref{sh}--(\ref{rh}), (\ref{boun}) and (\ref{ini}) is equivalent to the well-posedness of the system (\ref{4th2}), (\ref{shock2}), (\ref{boun3}), (\ref{boun21}) and (\ref{ini2}).To prove Theorem \ref{wpdthm} and obtain the well-posedness of the system of (\ref{arz}), \eqref{sh}--(\ref{rh}), (\ref{boun}) and (\ref{ini}), we just need to introduce the following lemma whose proof can be obtained by adaption of \cite[Appendix B]{BCBook}.
    \begin{lemma}\label{lm-ini}
		For each $T>0,$ there exists $\delta(T)>0$ s.t. for arbitrary $x_s^0\in(0,L),\ \mathbf{u}_0\in H^2((0,x_s^*);\mathbb{R}^4)$ satisfying the first-order compatibility condition and
		\begin{equation}\label{ineql1}
			\begin{aligned}
				|\mathbf{u}_0|_{H^2\left((0,x_s^*);\mathbb{R}^4\right)}+|x_s^0-x_s^*|&\leqslant\delta(T),
			\end{aligned}
		\end{equation}
		the system $(\ref{4th2}),\ (\ref{shock2}),\ (\ref{boun3}),\ (\ref{boun21})$ and $(\ref{ini2})$ has a unique solution $\mathbf{u}\in C^0\left([0,T];H^2\left((0,x_s^*);\mathbb{R}^4\right)\right)$, $x_s\in C^1\left([0,T]\right)$. Moreover, for any $t\in[0,T]$ it is valid that
		\begin{equation}\label{iniT}
        \begin{aligned}
			&|\mathbf{u}(t,\cdot)|_{H^2\left((0,x_s^*);\mathbb{R}^4\right)}+|x_s(t)-x_s^*|\\
            &\leqslant C(T)\left(|\mathbf{u}_0|_{H^2\left((0,x_s^*);\mathbb{R}^4\right)}+\left|x_s^0-x_s^*\right|\right).
            \end{aligned}
		\end{equation}
	\end{lemma}
    This completes the proof of Theorem \ref{thm02}.
	\end{proof}
    	\begin{remark}
		Obviously from $(\ref{boun21})$, the value of $\mathcal{B}$ only depends on $u_i(t,x_s^*),\ i=1,\ 2,\ 3,\ u_4(t,0)$ and $x_s-x_s^*$, which is consistent with the former conclusion that $u_4(t,x_s^*)$ can be regarded as a function of $u_i(t,x_s^*),\ i=1,\ 2,\ 3$ from $(\ref{boun3})$.
	\end{remark}
\section{The exponential stability of steady state in $H^2$ norm}
\label{stb-pf}
This section is devoted to the proof of our main result--Theorem \ref{stbthm}.

Firstly we should point out that, it is only necessary to prove the exponential stability of the null-steady state of the system (\ref{4th2}), (\ref{shock2}), (\ref{boun3}) and (\ref{boun21}) in $H^2$ as a result of the equivalence between the system (\ref{arz}), \eqref{sh}--(\ref{rh}) and (\ref{boun}) and the system (\ref{4th2}), (\ref{shock2}), (\ref{boun3}).
	
	Introduce the following Lyapunov functions
	\begin{equation}
		V(\mathbf{u},x_s)=\sum_{i=1}^3V_i(\mathbf{u})+\sum_{i=4}^6V_i(\mathbf{u},x_s),
		\label{Lv}
	\end{equation}
	where
	\begin{align}
    \label{Lv1}	
V_1=&\int_0^{x_s^*}\sum_{i=1}^4p_i\exp\left(-\frac{\mu x}{x_i\lambda_i}\right)u_i^2dx,\\	
\label{Lv2}	V_2=&\int_0^{x_s^*}\sum_{i=1}^4p_i\exp\left(-\frac{\mu x}{x_i\lambda_i}\right)u_{it}^2dx,\\
\label{Lv3}				V_3=&\int_0^{x_s^*}\sum_{i=1}^4p_i\exp\left(-\frac{\mu x}{x_i\lambda_i}\right)u_{itt}^2dx,
\end{align}
\begin{align}
\label{Lv4}	V_4=&\int_0^{x_s^*}\sum_{i=1}^3\frac{p_i'}{\lambda_i}\exp\left(-\frac{\mu x}{x_i\lambda_i}\right)(x_s-x_s^*)u_idx\\
&+C_0(x_s-x_s^*)^2,\nonumber\\
		\label{Lv5}		V_5=&\int_0^{x_s^*}\sum_{i=1}^3\frac{p_i'}{\lambda_i}\exp\left(-\frac{\mu x}{x_i\lambda_i}\right)\dot{x_s}u_{it}dx+C_0(\dot{x_s})^2,\\
		\label{Lv6}		V_6=&\int_0^{x_s^*}\sum_{i=1}^3\frac{p_i'}{\lambda_i}\exp\left(-\frac{\mu x}{x_i\lambda_i}\right)\ddot{x_s}u_{itt}dx+C_0(\ddot{x_s})^2,
	\end{align}
	where $p_i,\ C_0,\ p_i'$ are constants and $p_i>0,\ C_0>\frac32$.

    In the following we denote for simplicity
$|\mathbf{u}|_{H^2} := |\mathbf{u}(t, \cdot)|_{H^2((0, x^*_s); \mathbb{R}^4)}$
in the computations.
We propose the following lemma of the equivalence between the Lyapunov functions $V$ defined in (\ref{Lv})--(\ref{Lv6}) and $(|\mathbf{u}|_{H^2}+\left|x_s-x_s^*\right|)^2$.
	\begin{lemma}
		The Lyapunov functions $V$ defined in (\ref{Lv})--(\ref{Lv6}) is equivalent to the norm $(|\mathbf{u}|_{H^2}+\left|x_s-x_s^*\right|)^2$, if  $|\mathbf{u}|_{H^2}+\left|x_s-x_s^*\right|$ is small enough and the constants $p_i,\ p_i'$ satisfy that
		\begin{equation}
			\max_{i=1,2,3}\left\{\frac{p_i'^2x_i}{\mu\lambda_ip_i}\left(1-\exp\left(-\frac{\mu x_s^*}{x_i\lambda_i}\right)\right)\right\}<2.
			\label{const}
		\end{equation}
		\label{lm-eqv}
	\end{lemma}
The proof relies on using repeatedly Young's Inquality and Cauchy-Schwarz Inequality, we omit the details here.

It is worth noticing that we can separate the first three components of $\mathbf{u}$ under the feedback control condition (\ref{boun21}) that are actually controlled as {\color{black}$\mathbf{v}=(u_1,u_2,u_3)^\top$}, since from (\ref{boun3}) the boundary condition of $u_4$ at $x=x_s^*$ acquired by that of $u_1,\ u_2$ and $u_3$. From (\ref{boun21}) and the fact that $\mathcal{B}\in C^2$, we have
\begin{equation}
\label{vboun}
\begin{aligned}
\mathbf{v}(t,0) =& \partial_1\mathcal{B}(\mathbf{0},0,0)\mathbf{v}\left(t,x_s^*\right)
+\partial_3\mathcal{B}(\mathbf{0},0,0)(x_s-x_s^*) \\
&+ O\!\left((|\mathbf{u}|_{H^2}+|x_s-x_s^*|)^2\right).
\end{aligned}
\end{equation}
By direct computations, we get
\begin{equation}
	\partial_1\mathcal{B}(\mathbf{0},0,0)=K,\  \partial_3\mathcal{B}(\mathbf{0},0,0)=(b_1,b_2,b_3)^\top,
	\label{p1B}
\end{equation}
where $K$ and $(b_1,b_2,b_3)^{\top}$ are given in (\ref{stbnt-K})--(\ref{stbnt-b}).

For given $\bar{T}>0,$ let $x_s^0\in(0,L),\ \mathbf{u}^0\in H^2\left((0,x_s^*);\mathbb{R}^4\right)$ satisfy the first-order compatibility condition and (\ref{ineql1}), and the solution of system (\ref{4th2}), (\ref{shock2}), (\ref{boun3}), (\ref{boun21}) and (\ref{ini2}) is $\mathbf{u}\in C^0\left([0,\bar{T}];H^2\left((0,x_s^*);\mathbb{R}^4\right)\right),\ x_s\in C^1\left([0,\bar{T}]\right)$. First we suppose that $\mathbf{u}\in C^3$, using integration by parts, then
\begin{equation}
\begin{aligned}
	\frac{dV_1}{dt}=&-\mu V_1-\left[\sum_{i=1}^4p_ix_i\lambda_i\exp\left(-\frac{\mu x}{x_i\lambda_i}\right)u_i^2\right]_0^{x_s^*}\\
    &+O\left((|\mathbf{u}|_{H^2}+\left|x_s-x_s^*\right|)^3\right),
	\label{dv1}
    \end{aligned}
\end{equation}
Taking the derivative of (\ref{4th2}) and through integration by parts, we have
$$\begin{aligned}
	&\frac{dV_4}{dt}=\!-\mu\left(V_4\!-\!C_0(x_s-x_s^*)^2\right)\!+\!O\left((|\textbf{u}|_{H^2}+|x_s-x_s^*|)^3\right)\\
    &-(x_s-x_s^*)\left[\sum_{i=1}^3p_i'x_i\exp\left(-\frac{\mu x}{x_i\lambda_i}\right)u_i(t,x)\right]_0^{x_s^*}\\
	&+\!\dot{x_s}\!\left[\!2C_0(x_s\!-\!x_s^*)\!+\!\!\int_0^{x_s^*}\sum_{i=1}^3\frac{p_i'}{\lambda_i}\exp\left(\!\!-\!\frac{\mu x}{x_i\lambda_i}\right)u_i(t,x)dx\!\right],
\end{aligned}$$
  From (\ref{boun3}) and using the notations defined in \eqref{notations}, the linearization of (\ref{shock2}) can be written in the following compact form
\begin{equation*}
\dot{x_s}=\sum_{i=1}^3\left(\textcolor{black}{-\frac{\lambda_i}{\widetilde\lambda_i}}+\textcolor{black}{\frac{s_i\lambda_4}{\widetilde \lambda_4}}R_i\right)\frac{u_i(x_s^*)}{\rho_c^*-\rho_f^*}:= \sum_{i=1}^3\Theta_i u_i(x_s^*),
\end{equation*}
then we obtain
$$
\begin{aligned}
	&\quad \frac{d(V_1+V_4)}{dt}=-\mu\left(V_1+V_4-C_0(x_s-x_s^*)^2\right)\\
    &-\left[\sum_{i=1}^4p_ix_i\lambda_i\exp\left(-\frac{\mu x}{x_i\lambda_i}\right)u_i^2\right]_0^{x_s^*}\\
	&-(x_s-x_s^*)\left[\sum_{i=1}^3p_i'x_i\exp\left(-\frac{\mu x}{x_i\lambda_i}\right)u_i\right]_0^{x_s^*}\\
	&+\!\!\left(\!\sum_{i=1}^3\! \Theta_iu_i(x_s^*)\! \! \right)\! \! \Bigg(\! \!2C_0(x_s\!\! \!-\! x_s^*) \!\!+\! \! \!\int_0^{x_s^*}\!\! \sum_{i=1}^3\frac{p_i'}{\lambda_i}\exp\! \left(\! \!-\frac{\mu x}{x_i\lambda_i}\!\!\right)u_idx\! \!\Bigg)\\
    &+O\left((|\mathbf{u}|_{H^2}+|x_s-x_s^*|)^3\right).
\end{aligned}
$$
Using the boundary conditions \eqref{boun3}, (\ref{vboun}) and noticing \eqref{p1B}, we obtain that

\begin{equation*}
\begin{split}
&\quad \frac{d(V_1+V_4)}{dt}= -\mu(V_1+V_4)  + O\Big((|\mathbf{u}|_{H^2}+|x_s-x_s^*|)^3\Big)\\
&-\!\mathbf{v}^\top\!(x_s^*)\big(F(x_s^*,\mu)\!-\!K^\top F(0,\mu)K\big)\mathbf{v}(x_s^*) \!+\! p_4x_4\lambda_4u_4^2(0)\\
&- p_4x_4\lambda_4\exp\!\left(-\frac{\mu x_s^*}{x_4\lambda_4}\right)\left(\sum_{i=1}^3 R_i s_i u_i(x_s^*)\right)^2 \\
&+ \!\!\mu C_0(x_s\!\!-\!\!x_s^*)^2\!\!\sum_{i=1}^3 \!p_i x_i \lambda_i b_i^2\!\! +\!\! 2\sum_{i=1}^3 p_i x_i \lambda_i b_i (x_s\!\!-\!x_s^*)\!\!\!\sum_{j=1}^3\! k_{ij}u_j(x_s^*) \\
&+ (x_s-x_s^*)^2\sum_{i=1}^3 p_i' x_i b_i \!\!-\!\! (x_s-x_s^*)\sum_{i=1}^3 p_i' x_i u_i(x_s^*)\exp\!\left(\!\!\!-\frac{\mu x_s^*}{x_i\lambda_i}\right)\\
& + (x_s-x_s^*)\sum_{i=1}^3 p_i' x_i \sum_{j=1}^3 k_{ij}u_j(x_s^*)+ 2C_0(x_s-x_s^*)\sum_{i=1}^3 \Theta_i u_i(x_s^*) \\
&\quad + \left(\sum_{i=1}^3 \Theta_i u_i(x_s^*)\right)\int_0^{x_s^*}\sum_{j=1}^3 \frac{p_j'}{\lambda_j}\exp\!\left(-\frac{\mu x}{x_j\lambda_j}\right)u_j dx,
\end{split}
\end{equation*}
where $F(x,\mu)\! =\! \mbox{diag}\left(\lambda_ip_ix_i\exp\left(\! -\frac{\mu x}{x_i\lambda_i}\right)\right)_{i\in\{1,2,3\}}$. Noting that except for the last term, other terms are all quadratic forms of $\left(\mathbf{v}^\top(x_s^*),u_4(0),x_s-x_s^*\right)$. Using Young's Inequality and Cauchy-Schwarz Inequality successively, the following holds
\begin{align*}
&\left(\sum_{i=1}^3\Theta_iu_i(x_s^*)\right)\int_0^{x_s^*}\frac{p_j'}{\lambda_j}\exp\left(\! -\frac{\mu x}{x_j\lambda_j}\right)u_jdx\\
&\leqslant \! \frac{1}{2\varepsilon_j}\left(\sum_{i=1}^3\Theta_iu_i(x_s^*)\right)^2\!\!\!+\!\!\frac{\varepsilon_j}2\! \left(\int_0^{x_s^*}\! \frac{p_j'}{\lambda_j}\exp\left(\! -\frac{\mu x}{x_j\lambda_j}\right)u_jdx\right)^2\\
&\leqslant\frac{\varepsilon_j}{2\mu}\left[\frac{p_j'^2x_j}{p_j\lambda_j}\left(\!\!1\!-\!\exp\left(\!-\frac{\mu x_s^*}{x_j\lambda_j}\!\right)\!\right)\!\right]\!\int_0^{x_s^*}\!\!p_j\!\exp\left(\!\!-\frac{\mu x}{x_j\lambda_j}\right)u_j^2dx\\
	&+\frac{1}{2\varepsilon_j}\left[\sum_{i=1}^3\Theta_i^2u_i^2(x_s^*)\left(\sum_{k=1}^3\exp\left(\frac{\mu x_s^*}{x_i\lambda_i}-\frac{\mu x_s^*}{x_k\lambda_k}\right)\right)\right].
\end{align*}
for each $j=1,2,3$.
Above all, we have
\begin{align*}
	&\frac{d(V_1+V_4)}{dt}\leqslant\! -\! \mu\! \left(V_1\! +\! V_4\right)\!+O\left((|\mathbf{u}|_{H^2}+|x_s-x_s^*|)^3\right) \\
    &-\mathbf{v}^{\top}\! (x_s^*)\! \Bigg(\! \! F(x_s^*,\mu)\! -\! K^\top\!  F(0,\mu)K\! -\! \frac{\sum_{k=1}^3\frac{1}{\varepsilon_k}}{2(\rho_c^*-\rho_f^*)^2}\! \widetilde{D}(\mu)\! \Bigg)\! \mathbf{v}\! (x_s^*)\\
    &+\! p_4x_4\lambda_4u_4^2(0)-p_4x_4\lambda_4\exp\left(-\frac{\mu x_s^*}{x_i\lambda_i}\right)\left(\sum_{i=1}^3R_is_iu_i(x_s^*)\right)^2\\
    &+(x_s-x_s^*)^2\Bigg(\mu C_0+\sum_{i=1}^3x_ib_i(p_i\lambda_i+p_i')\Bigg)\notag\\
	&+(x_s-x_s^*)\sum_{j=1}^3\! \! \Bigg[\! 2C_0\theta_j\! -p_j'x_j\exp\! \left(\! -\frac{\mu x_s^*}{x_j\lambda_j}\right)\\
    &\quad\quad\quad+\! \! \sum_{i=1}^3\! x_ik_{ij}\! \left(2p_i\lambda_ib_i\! +\! p_i'\right)\!  \Bigg]u_j(x_s^*)\notag\\
	&+\! \sum_{i=1}^3\! \left(\! 1\! -\! \exp\! \left(\! -\frac{\mu x_s^*}{x_i\lambda_i}\right)\! \right)\frac{\varepsilon_ip_i'^2x_i}{2\mu\lambda_ip_i}\int_0^{x_s^*}p_i\exp\left(-\frac{\mu x}{x_i\lambda_i}\right)u_i^2dx.
\end{align*}

To establish an exponential decay, denote by
\begin{equation}
\label{eps}
\frac1{\varepsilon_i}=\frac{p_i'^2x_i\left(1-\exp\left(-\frac{\mu x_s^*}{x_i\lambda_i}\right)\right)}{\mu^2\lambda_ip_i},\quad i=1,2,3.
\end{equation}
Then we obtain that
\begin{equation}
	\begin{aligned}
		&\frac{d(V_1+V_4)}{dt}\leqslant\! -\! \frac{\mu}2V_1\! -\! \mu V_4+O\left((|\mathbf{u}|_{H^2}+|x_s-x_s^*|)^3\right)\\
        &-\! \mathbf{v}^{\top}\! (x_s^*)\! \left(\! F(x_s^*,\mu)\! -\! K^\top F(0,\mu)K\! -\! \frac{\sum_{k=1}^3\frac{1}{\varepsilon_k}}{2(\rho_c^*\! -\! \rho_f^*)^2}\widetilde{D}(\mu)\! \right)\mathbf{v}(x_s^*)\\
		&+p_4x_4\lambda_4u_4^2(0)-p_4x_4\lambda_4\exp\left(\! -\frac{\mu x_s^*}{x_i\lambda_i}\right)\left(\sum_{i=1}^3R_is_iu_i(x_s^*)\right)^2\\
		&+(x_s-x_s^*)^2\Bigg(\mu C_0+\sum_{i=1}^3x_ib_i(p_i\lambda_i+p_i')\Bigg)\\
        &+(x_s-x_s^*)\sum_{j=1}^3\! \Bigg(\! 2C_0\Theta_j\! -\! p_j'x_j\\
		&+\! \! \sum_{i=1}^3\! x_ik_{ij}\! \left(2p_i\lambda_ib_i\! +\! p_i'\right)\exp\! \left(\! -\frac{\mu x_s^*}{x_j\lambda_j}\right)\! \! \Bigg)u_j(x_s^*).
	\end{aligned}
	\label{dv14-eps}
\end{equation}

It turns out that each term in (\ref{dv14-eps}) is a quadratic form of $\left(\mathbf{v}^\top(x_s^*),u_4(0),x_s-x_s^*\right)$, except the first two terms for exponential decay. Take
\begin{equation}
	p_i'=\frac{2C_0\Theta_i}{x_i}\exp\! \left(\! \frac{\mu x_s^*}{x_i\lambda_i}\! \right)\! ,\ i=1,2,3.
	\label{pi'}
\end{equation}
for further simplicity, and from (\ref{bi+})--(\ref{bi-}) it comes out that $b_ix_ip_i'<0,\ i=1,2,3.\ $ for arbitrary $\gamma>0$. Thus,
\begin{equation}
	p_i=-\frac{p_i'}{2\lambda_ib_i}>0,\quad i=1,2,3.
	\label{pi}
\end{equation}
Plug (\ref{pi'})--(\ref{pi}) into (\ref{dv14-eps}) and take Young's inequality,
\begin{equation}
\begin{aligned}
&\frac{d(V_1+V_4)}{dt} \leqslant -\frac{\mu}{2}V_1 - \mu V_4 +O\left((|\mathbf{u}|_{H^2}+|x_s-x_s^*|)^3\right)\\
&- \mathbf{v}^\top(x_s^*) \Bigg[ F(x_s^*,\mu) - K^\top F(0,\mu)K -\frac{\sum_{k=1}^3\frac{1}{\varepsilon_k}}{2(\rho_c^* - \rho_f^*)^2}\widetilde{D}(\mu)\\
&-\mathrm{diag}\left(3p_4x_4\lambda_4R_i^2\exp\left(\frac{\mu x_s^*}{x_i\lambda_i}\right)\right)_{i\in\{1,2,3\}}\Bigg]\mathbf{v}(x_s^*) \\
& + p_4x_4\lambda_4u_4^2(0) + (x_s-x_s^*)^2\left(\mu C_0+\frac12\sum_{i=1}^3x_ib_ip_i'\right).
\end{aligned}
\label{dv14-als}
\end{equation}

Since the inequality in (\ref{bi+})--(\ref{bi-}) is strict, there exists $\mu>\gamma$ such that the inequality (\ref{dv14-als}) also checks out to $\mu$. Take such $\mu$ and use (\ref{pi'}),
\begin{equation}
	\mu C_0+\frac12\sum_{i=1}^3x_ib_ip_i'<0.
	\label{ineq^2}
\end{equation}
Moreover, \textcolor{black}{ by (\ref{p1B}), (\ref{eps}), (\ref{pi'})--(\ref{pi})} and the positive definite condition of matrix (\ref{+dft}), we have that the matrix
$$F(x_s^*,\mu)-K^\top F(0,\mu)K-\frac{\left(\sum_{k=1}^3\frac{1}{\varepsilon_k}\right)}{2\left(\rho_c^*-\rho_f^*\right)^2}\widetilde{D}(\mu)$$
is positive definite. Therefore, there exists $p_4>0$ that the quadratic form of $\mathbf{v}(x_s^*)$ in (\ref{dv14-als}) satisfying
\begin{equation*}
\begin{aligned}
&\mathbf{v}^{\top}(x_s^*)\Bigg[F(x_s^*,\mu)\! -K^\top F(0,\mu)K-\frac{\mu\sum_{k=1}^3\frac{1}{\varepsilon_k}}{2\left(\rho_c^*-\rho_f^*\right)^2}\\
	&\, \widetilde{D}-\text{diag}\left(3p_4x_4\lambda_4R_i^2\exp\left(\frac{\mu x_s^*}{x_i\lambda_i}\right)_{i\in\{1,2,3\}}\right)\Bigg]\mathbf{v}(x_s^*)<0,
\end{aligned}
\end{equation*}
and
\begin{equation}
	\frac{d(V_1+V_4)}{dt}\!\!\leqslant\!\!-\!\frac{\mu}2V_1\!-\!\mu V_4\!+\!O\left((|\mathbf{u}|_{H^2}\!+\!|x_s\!-\!x_s^*|)^3\right).
\end{equation}
Since $\mu>\gamma$, from Lemma \ref{lm-ini}, it can be ensured that $|\mathbf{u}|_{H^2}+|x_s-x_s^*|$ approaches to $0$ infinitely if $\delta\left(\bar{T}\right)$ is small enough. Thus,
\begin{equation}
	\frac{d(V_1+V_4)}{dt}\leqslant-\frac{\gamma}2(V_1+ V_4).
\end{equation}
With the same process omitted here, we can also derive $\dot V_2$, $\dot V_3$, $\dot V_5$ and $\dot V_6$.
It is worth noticing that from (\ref{Lv1})--(\ref{Lv6}), both $V_2+V_5$ and $V_3+V_6$ share a similar structure to $V_1+V_4$, so we
omit the details for the estimation of $V_2+V_5$ and $V_3+V_6$ here.

Under the assumption that each of (\ref{4th2}), (\ref{shock2}), (\ref{boun3}) and (\ref{boun21}) is $C^3$, we obtain
\begin{equation}
	\frac{dV}{dt}\leqslant-\frac{\gamma}2V.
	\label{dvdc}
\end{equation}
By using the density argument similarly as in \cite{burger}, and since $\gamma$ is independent of the $C^2$ or $C^3$ norm of $\mathbf{u}$, (\ref{dvdc}) also satisfies for arbitrary
$$(\mathbf{u},x_s)\in C^0\left(\left[0,\Bar{T}\right];H^2\left((0,x_s^*);\mathbb{R}^4\right)\right)\times C^1\left(\left[0,\Bar{T}\right];\mathbb{R}\right).$$
By Lemma \ref{lm-eqv}, the Lyapunov functions defined by (\ref{Lv})-(\ref{Lv6}) is equivalent with $\left(|\mathbf{u}|_{H^2}+|x_s-x_s^*|\right)^2$, when  $\left(|\mathbf{u}|_{H^2}+|x_s-x_s^*|\right)^2$ is small enough, and now we have already obtained that the null-steady state of the system (\ref{4th2}), (\ref{shock2}), (\ref{boun3}) and (\ref{boun21}) decays in a rate of $\frac{\gamma}4$ in $H^2$ norm. It is essential to check whether the assumptions (\ref{const}) and (\ref{bi+})-(\ref{bi-}) is valid under the values of (\ref{pi'}) and (\ref{pi}). From (\ref{pi'}), (\ref{pi}),
$$\max_{i}\left\{\frac{p_i'x_i}{\mu\lambda_ip_i}\left(1-\exp\left(-\frac{\mu x}{x_i\lambda_i}\right)\right)\right\}<\frac{4C_0}3,$$
therefore take $C_0<\frac32$ then (\ref{const}).

We notice that up to now $\delta(\Bar{T})$ is dependent on $\Bar{T}$, compared to the Definition \ref{stbdef}, it remains to prove that for arbitrary $T>0$, there exists a consistent $\delta^*>0$ such that (\ref{dvdc}) holds on $(0,T)$. Assume that $x_s^0\in(0,L)$, $\mathbf{u}^0\in H^2\left((0,x_s^*);\mathbb{R}^4\right)$ both satisfy the first-order compatibility condition, and
\begin{equation}\label{dlt0-1}	\left|\mathbf{u}^0\right|_{H^2}+\left|x_s^0-x_s^*\right|<\Bar{\rho},\ V\left(\mathbf{u}^0,x_s^0\right)\leqslant\nu
\end{equation}
for any given $\nu>0$. Take $\nu$ small enough that for any $t\in\left[0,\Bar{T}\right]$, with (\ref{iniT}), (\ref{dvdc}) and by the equivalence of $V$ and the norm
$|\mathbf{u}|_{H^2}+\left|x_s-x_s^*\right|$, we have
\begin{equation}\label{dlt1T-1}
	|\mathbf{u}|_{H^2}+\left|x_s-x_s^*\right|<\Bar{\rho},\ V\left(\mathbf{u}(t),x_s(t)\right)\leqslant\nu.
\end{equation}
It can be observed that \eqref{dlt1T-1} also hold for any $t\in\left[\Bar{T},2\Bar{T}\right]$. Thus we deduce that
\begin{equation}
	|\mathbf{u}|_{H^2}+\left|x_s-x_s^*\right|<\Bar{\rho},\quad t\in\left[(j-1)\Bar{T},j\Bar{T}\right],
	\label{dltjT-1}
\end{equation}
and
\begin{equation}
	V\left(\mathbf{u}(t),x_s(t)\right)\leqslant\nu,\quad t\in\left[(j-1)\Bar{T},j\Bar{T}\right],
	\label{dltjT-2}
\end{equation}
And in the sense of distribution,
\begin{equation}
	\frac{dV}{dt}\leqslant-\frac{\gamma}2V,\quad t\in\left(0,j\Bar{T}\right).
	\label{dltjT-3}
\end{equation}
From (\ref{rco}), there exists $\delta^*$ such that (\ref{dlt0-1}) can be derived if (\ref{wpdcon1}) hold, also, for arbitrary $T>0$ a corresponding $j\in\mathbb{N}$ exists that $\left(0,T\right)\subset\left(0,j\Bar{T}\right)$. Thus, the steady state $\left(\left(\rho^*,z^*\right)^\top,x_s^*\right)$ locally exponentially decay with a rate of $\frac{\gamma}4$ in $H^2$ norm. This completes the proof of Theorem \ref{stbthm}.

\begin{remark}[Existence of the tuning parameters]\label{exK}
	Assume the matrix $K$ is diagonal, i.e. $K=\mbox{diag}(k_1,k_2,k_3)$, then matrix (\ref{+dft}) can be converted into $\mbox{diag}(D_i,\ i=1,2,3)$, where
	$$\begin{aligned}
		D_i&=\frac{-L_i\left(\rho_c^*\! -\! \rho_f^*\right)}{b_i}-\frac{-L_i\left(\rho_c^*\! -\! \rho_f^*\right)}{b_i}\exp\! \left(\! \frac{\gamma x_s^*}{x_i\lambda_i}\right)k_i^2\\
		&+\left[\! \sum_{k=1}^3\! \frac{2b_kL_k\left(\rho_c^*\! -\! \rho_f^*\right)}{\gamma^2}\left(\! \exp\! \left(\! \frac{\gamma x_s^*}{x_k\lambda_k}\right)\! -1\! \right)\! \right]\times\\
       &\qquad\qquad\qquad\qquad  L_i^2\left(\sum_{j=1}^3\exp\left(\frac{\gamma x_s^*}{x_i\lambda_i}-\frac{\gamma x_s^*}{x_j\lambda_j}\right)\right).
	\end{aligned}$$
	Note that we only need to prove
	\begin{equation}
		k_i^2<\exp\left(\frac{-\gamma x_s^*}{x_i\lambda_i}\right)K_i, \ i=1,2,3.
	\end{equation}
	to prove that the matrix (\ref{+dft}) is positive definite, where
	\begin{equation}
    \begin{aligned}
		K_i=&1\! -b_iL_i \Bigg[\! \sum_{i=1}^3\frac{2b_kL_k}{\gamma^2}\left(\! \exp\! \left(\frac{\gamma x_s^*}{x_k\lambda_k}\right)\! -1\! \right)\! \\
        &\quad\qquad\qquad\left(\sum_{j=1}^3\exp\! \left(\frac{\gamma x_s^*}{x_i\lambda_i}-\frac{\gamma x_s^*}{x_j\lambda_j}\! \right)\! \right)\! \Bigg].
		\label{Ki}
        \end{aligned}
	\end{equation}
	Plug the limit case in (\ref{bi+})--(\ref{bi-}), which is $b_i=\frac{-\gamma\exp\left(-\frac{\gamma x_s^*}{x_i\lambda_i}\right)}{3L_i}$, into (\ref{Ki}), we have
	\begin{equation}
		K_i\!=\!1\!-\!\frac29\left[\!3\!\left(\!\sum_{k=1}^3\!\exp\!\left(\!\frac{\gamma x_s^*}{x_k\lambda_k}\!\right)\!\!\right)\!\!-\!\!\left(\!\sum_{k=1}^3\!\exp\!\left(\!\frac{\gamma x_s^*}{x_k\lambda_k}\!\right)\!\!\right)^2\right].
		\label{Kisimp}
	\end{equation}
	Denote by $\displaystyle y=\sum_{k=1}^3\exp\left(\frac{\gamma x_s^*}{x_k\lambda_k}\right),$
	and (\ref{Kisimp}) is equivalent to the quadratic equation
	$\frac29y^2-\frac23y+1,$
	whose discriminant keeps negative and $K_i$ is always greater than $0$. Therefore, the corresponding $k_1,\ k_2,\ k_3$ and $b_1,\ b_2,\ b_3$ defined by (\ref{stbnt-K})--(\ref{stbnt-b}) can be found such that (\ref{bi+})--(\ref{bi-}) holds and the positive definite matrix (\ref{+dft}) exists.
\end{remark}

\section{Simulation}
\label{sec_simu}
We simulate the movement of the shock as well as the distribution of density in a highway segment. Take the length of the highway segment as $L=500$ m, and the road capacity is set as $\rho_m=180$ veh/km. The initial traffic profile $x_s(0)=200$ m, $\rho_f(0,x)=65\ \mbox{veh/km},\ \rho_c(0,x)=130\ \mbox{veh/km}, x\in[0,L]$. The desirable steady states of traffic flow are as follows: $x_s^*=120$ m, $\rho_f^*=60$ veh/km, $\rho_c^*=150$ veh/km, $z_f^*=220$ veh/s, $z_f^*=587.5$ veh/s. Take the traffic pressure as $p(\rho)=24.5(\frac{\rho}{\rho_{m}}-1).$ The control input $\rho(t,0)$ illustrated in Fig.\ref{fig:controlInput} is regulated from the initial data to the desired steady state $\rho^*_f$ under closed-loop operation, while deviates away from $\rho^*_f$ under open-loop operation.
The system can be regulated from the initial profile to the
reference profile using boundary feedback control \eqref{boun}, 
as shown in Fig.\ref{fig:characteristic-line}. The instability of the open-loop system is further illustrated in Fig.\ref{fig:norm}, where the norm of the deviation of the total density from equilibrium increases under open-loop operation, while it converges to zero in closed-loop operation.

One can note that in the absence of control, the open-loop case exhibits upstream shock propagation, leading to the continuous spread of congestion. In contrast, the shock trajectory is regulated and converges towards the prescribed location when the boundary feedback control is applied. After a transient phase, the density approaches the given steady state characterized by two constant states separated by a stationary interface, which is exactly the shock.
\begin{figure}[htbp]
  \centering
  \includegraphics[width=1\linewidth]{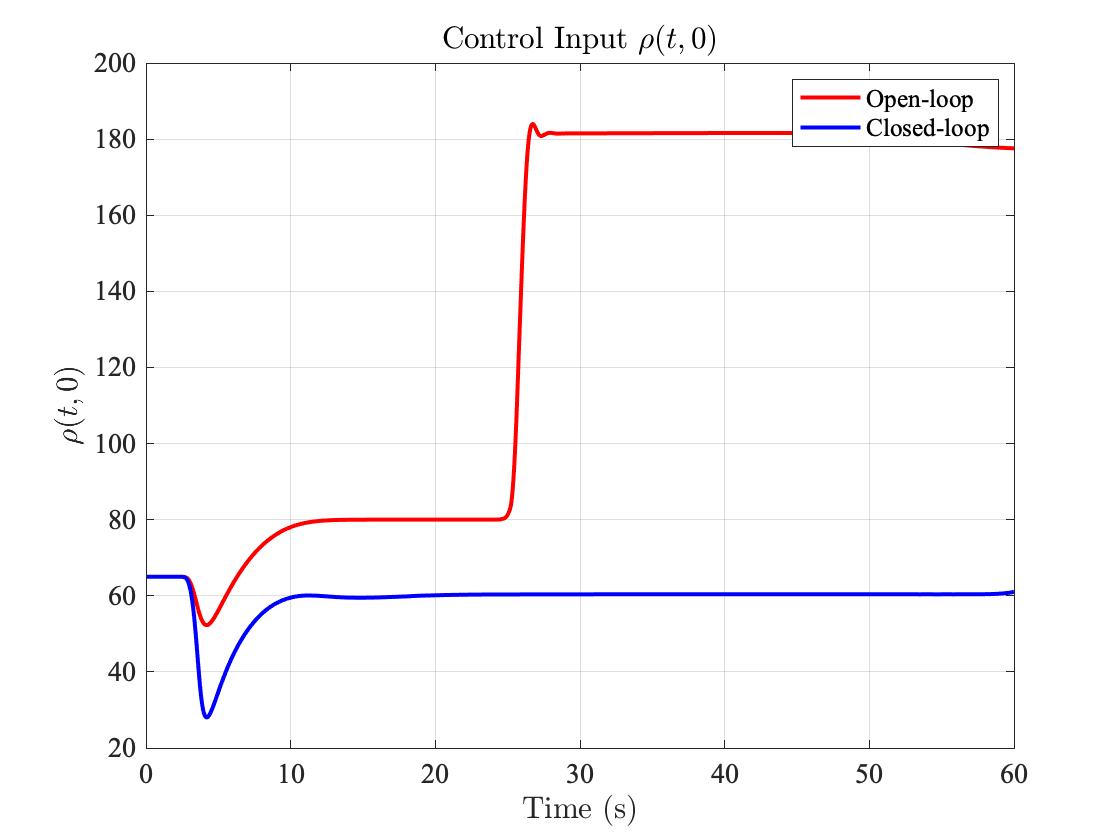}
  \caption{Control input $\rho(t,0)$ for the ARZ model under open-loop and closed-loop cases. } 
  \label{fig:controlInput}
\end{figure}

\begin{figure}[htbp]
  \centering
  \includegraphics[width=1\linewidth]{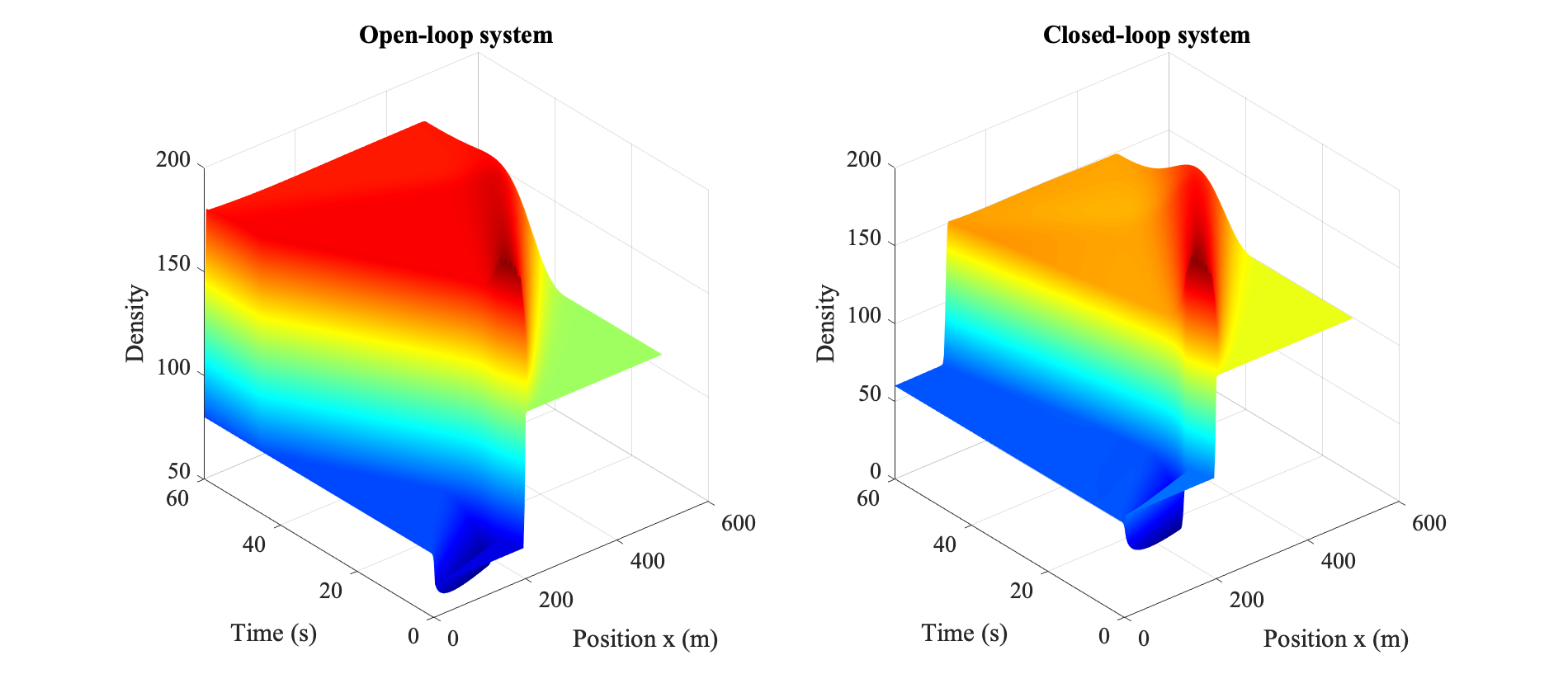}
  \caption{Evolution of traffic density for the ARZ model under both open-loop and closed-loop cases. } 
  \label{fig:characteristic-line}
\end{figure}

\begin{figure}[htbp]
  \centering
  \includegraphics[width=1\linewidth]{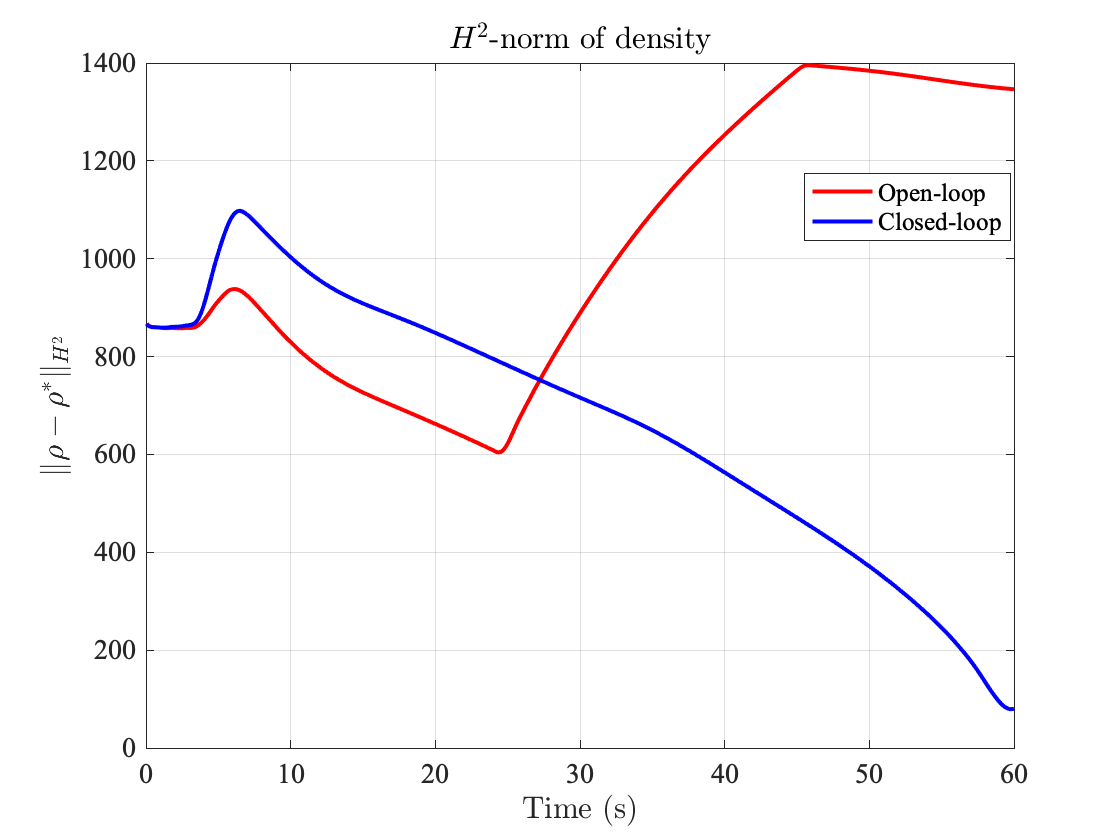}
  \caption{Evolution of $H^2$-norm of the total density } 
  \label{fig:norm}
\end{figure}

\section{Conclusions}
\label{sec_con}
This paper studies boundary feedback stabilization of traffic flow with shocks in the ARZ model. By designing a shock-dependent boundary feedback law, the number of shock waves is prevented from increasing and the system energy  is shown to decay exponentially at an arbitrary rate. To establish stability, a shock-dependent perturbation term is incorporated into the Lyapunov function, modifying its classical structure. Analysis of the Lyapunov evolution proves $H^2$-exponential energy decay. Numerical simulations based on the finite difference method further illustrate the effectiveness of the proposed feedback control.

\bibliographystyle{IEEEtran}
\bibliography{refs_ieee}

@book {BCBook,
    AUTHOR = {Bastin, Georges and Coron, Jean-Michel},
     TITLE = {Stability and Boundary Stabilisation of 1-D Hyperbolic Systems},
    SERIES = {Number 88 in Progress in Nonlinear Differential Equations and Their Applications},
    VOLUME = {},
      NOTE = {},
 PUBLISHER = {Springer International},
      YEAR = {2016},
}

@INPROCEEDINGS{coron1st1,
  author={Coron, J. M. and d'Andréa-Novel, B. and Bastia, G.},
  booktitle={1999 European Control Conference (ECC)}, 
  title={A Lyapunov approach to control irrigation canals modeled by saint-venant equations}, 
  year={1999},
  volume={},
  number={},
  pages={3178-3183}
  }

@ARTICLE{boun1,
  author={Boskovic, D.M. and Krstic, M. and Weijiu Liu},
  journal={IEEE Transactions on Automatic Control}, 
  title={Boundary control of an unstable heat equation via measurement of domain-averaged temperature}, 
  year={2001},
  volume={46},
  number={12},
  pages={2022-2028},
  doi={10.1109/9.975513}}

@book{boun3,
author = {Krstic, Miroslav},
title = {Boundary Control of {PDE}s: A Course on Backstepping Designs},
year = {2008},
isbn = {0898716500},
publisher = {Society for Industrial and Applied Mathematics},
address = {USA}
}

@ARTICLE{lwr,
  author={Lighthill, Michael James and Whitham, Gerald Beresford},
  journal={Proceedings of the Royal Society of London. Series A. Mathematical and Physical Sciences}, 
  title={On kinematic waves II. A theory of traffic flow on long crowded roads}, 
  year={1955},
  volume={229},
  number={1178},
  pages={317-345}
  }

@ARTICLE{richa,
  author={Richards, Paul I},
  journal={Operations Research}, 
  title={Shock Waves on the Highway}, 
  year={1956},
  volume={4},
  number={1},
  pages={42-51}
  }

@article{ar,
  author = {Aw, A. and Rascle, M.},
  title = {Resurrection of ``{S}econd {O}rder" Models of Traffic Flow},
  journal = {SIAM Journal on Applied Mathematics},
  volume = {60},
  number = {3},
  pages = {916-938},
  year = {2000}
}

@article {2017DDTK,
    AUTHOR = {Diagne, Ababacar and Diagne, Mamadou and Tang, Shuxia and
              Krstic, Miroslav},
     TITLE = {Backstepping stabilization of the linearized {\it
              {S}aint-{V}enant-{E}xner} model},
   JOURNAL = {Automatica J. IFAC},
  FJOURNAL = {Automatica. A Journal of IFAC, the International Federation of
              Automatic Control},
    VOLUME = {76},
      YEAR = {2017},
     PAGES = {345--354},
      ISSN = {0005-1098,1873-2836},
   MRCLASS = {93C20 (93D05)},
  MRNUMBER = {3590585},
}

@article{zg,
  title = {A non-equilibrium traffic model devoid of gas-like behavior},
  journal = {Transportation Research Part B: Methodological},
  volume = {36},
  number = {3},
  pages = {275-290},
  year = {2002},
  issn = {0191-2615},
  author = {H.M. Zhang}
}

@article {krstic20082,
    AUTHOR = {Krstic, Miroslav and Smyshlyaev, Andrey},
     TITLE = {Backstepping boundary control for first-order hyperbolic
              {PDE}s and application to systems with actuator and sensor
              delays},
   JOURNAL = {Systems Control Lett.},
  FJOURNAL = {Systems \& Control Letters},
    VOLUME = {57},
      YEAR = {2008},
    NUMBER = {9},
     PAGES = {750--758},
      ISSN = {0167-6911},
   MRCLASS = {93D15 (93C20)},
  MRNUMBER = {2446460},
MRREVIEWER = {Marcelo M. Cavalcanti},
}

@book {KrsticBook,
    AUTHOR = {Krstic, Miroslav and Smyshlyaev, Andrey},
     TITLE = {Boundary {C}ontrol of {PDE}s: A {C}ourse on {B}ackstepping {D}esigns},
    SERIES = {Advances in Design and Control},
    VOLUME = {16},
      NOTE = {},
 PUBLISHER = {Society for Industrial and Applied Mathematics (SIAM),
              Philadelphia, PA},
      YEAR = {2008},
     PAGES = {x+192},
   MRCLASS = {93C20 (35B37 74M05 76D55 93C40)},
  MRNUMBER = {2412038},
}

@inbook{rh1,
author={Rankine, W. J. Macquorn},
title={On The Thermodynamic Theory of Waves of Finite Longitudinal Disturbance},
bookTitle={Classic Papers in Shock Compression Science},
year={1998},
publisher={Springer New York},
address={New York, NY},
pages={133--148},
doi={10.1007/978-1-4612-2218-7_5},
url={https://doi.org/10.1007/978-1-4612-2218-7_5}
}

@article{burger,
author = {Bastin, Georges and Coron, Jean-Michel and Hayat, Amaury and Shang, Peipei},
title = {Exponential boundary feedback stabilization of a shock steady state for the inviscid Burgers equation},
journal = {Mathematical Models and Methods in Applied Sciences},
volume = {29},
number = {02},
pages = {271-316},
year = {2019}
}

@article{sv,
title = {Boundary feedback stabilization of hydraulic jumps},
journal = {IFAC Journal of Systems and Control},
volume = {7},
pages = {100026},
year = {2019},
issn = {2468-6018},
author = {Georges Bastin and Jean-Michel Coron and Amaury Hayat and Peipei Shang}}

@ARTICLE{bstp,
  author={Yu, Huan and Diagne, Mamadou and Zhang, Liguo and Krstic, Miroslav},
  journal={IEEE Transactions on Automatic Control}, 
  title={Bilateral Boundary Control of Moving Shockwave in {L}{W}{R} Model of Congested Traffic}, 
  year={2021},
  volume={66},
  number={3},
  pages={1429-1436}
  }

@ARTICLE{stopandgo,
  author={Zhang, Liguo and Luan, Haoran and Zhan, Jingyuan},
  journal={IEEE Transactions on Automatic Control}, 
  title={Stabilization of Stop-and-Go Waves in Vehicle Traffic Flow}, 
  year={2024},
  volume={69},
  number={7},
  pages={4583-4597}}

@mastersthesis{VARZ,
  author       = {Villa, Stefano},
  title        = {The {Aw-Rascle-Zhang} model with constraints},
  school       = {Universit{\`a} degli Studi di Milano-Bicocca},
  year         = {2016},
  type         = {Master's Thesis},
  eprint       = {1605.00632},
  archivePrefix = {arXiv},
  primaryClass = {math.AP},
  doi          = {10.48550/arXiv.1605.00632},
  url          = {https://arxiv.org/abs/1605.00632}
}

@article{YU2019,
title = {Traffic congestion control for {Aw–Rascle–Zhang }model},
journal = {Automatica},
volume = {100},
pages = {38-51},
year = {2019},
issn = {0005-1098},
author = {Huan Yu and Miroslav Krstic}
}

\end{document}